\documentclass[preprint,12pt]{elsarticle}

\usepackage{amssymb}
\usepackage{amsmath}
\PassOptionsToPackage{table}{xcolor}

\usepackage{hyperref}
\usepackage{natbib}
\usepackage{epstopdf}
\usepackage{url}
\usepackage[titletoc]{appendix}
\usepackage{eurosym}
\usepackage{listings}
\usepackage{setspace}
\usepackage{float}
\usepackage{booktabs}
\usepackage{pdfpages}
\usepackage{array}
\usepackage{placeins}
\usepackage{amsthm} 
\usepackage{graphicx}
\usepackage{mathrsfs}

\theoremstyle{definition}

\hypersetup{
    colorlinks,
    citecolor=black,
    filecolor=black,
    linkcolor=black,
    urlcolor=black
}

\journal{arxiv.org}

\begin{document}

\newtheorem{theorem}{Theorem}
\newtheorem{corollary}{Corollary}
\newtheorem{lemma}{Lemma}
\newtheorem{example}{Example}
\theoremstyle{definition}
\newtheorem{definition}{Definition}

\begin{frontmatter}



\title{On Some Types of Limit}


\author[BEU]{Ufuk Kaya}

\ead{mat-ufuk@hotmail.com}

\author[BEU]{Gökhan Turan}

\affiliation[BEU]{Department of Mathematics, Faculty of Arts and Sciences, Bitlis Eren University, Bitlis, Turkey}

\begin{abstract}
In this paper, we consider the concept of limit, one of the basic concepts of mathematical analysis. At a point $a\in{\mathbb{R}}$, the limit of a function $f$ from $A\subset\mathbb{R}$ to $\mathbb{R}$ is $L\in{\mathbb{R}}$ if and only if there exists $\delta>0$ such that the set
$$
\left\{x\in\left(\left(a-\delta,a+\delta\right)\backslash\left\{a\right\}\right)\cap A:\lvert f\left(x\right)-L\rvert \ge\varepsilon\right\}
$$
is the empty set for each $\varepsilon>0$. This study investigates what happens when the above set is not empty. More clearly, when we take the above set as finite or a set whose accumulation points is empty, we prove that the limit concept does not change. However, when we take the above set as countable or a set of measure zero, we prove that they are new limit concepts. Also, we show that these new limit concepts have the properties of existence, uniqueness, additivity, multiplicativity, etc. Finally, by giving some examples, we compare the limit concepts we have obtained with the previous ones.
\end{abstract}

\begin{keyword}
Limit \sep Approximate limit \sep Lebesgue measure

\MSC[2020]{26A03 \sep 26C06 \sep 28A20}




\end{keyword}

\end{frontmatter}


\section{Introduction}\label{sec1}

Limit is one of the most basic concepts of mathematics. Grégoire de Saint-Vincent was the first to mention the concept of limit in 1647. He use the concept for the convergence of geometric series.

Bernard Bolzano, who first proposed the fundamentals of the epsilon-delta technique to describe continuous functions in 1817, is credited with developing the contemporary definition of a limit. But, during his lifetime, his work was unknown.

The ($\varepsilon,\delta$)-definition of limit, which was defined by Karl Weierstrass and Augustin-Louis Cauchy in 1821, is a definition of the limit of a function.

In his book A Course of Pure Mathematics (1908), G. H. Hardy was the first to use the modern notation of placing the arrow below the limit symbol.

Below is the modern definition of the limit at a point $a\in\mathbb{R}$ of a function from $A\subset\mathbb{R}$ to $\mathbb{R}$:
For every real $\varepsilon>0$, there exists a real $\delta>0$ such that for all real $x\in A$, $0<\lvert x-a\rvert<\delta$  implies $\lvert f\left(x\right)-L \rvert<\varepsilon$.

According to the definition of the usual limit, the necessary and sufficient condition for the limit of a function from $A\subset\mathbb{R}$ to $\mathbb{R}$ to be $L\in\mathbb{R}$ at a point $a\in\mathbb{R}$ is there exists $\delta>0$ such that the set
\begin{equation}\label{eq1}
\mathcal{L}\left(f,A,a,L,\delta,\varepsilon\right)=\left\{x\in\left(\left(a-\delta,a+\delta\right)\backslash\left\{a\right\}\right)\cap A:\left|f\left(x\right)-L\right|\geq\varepsilon\right\}
\end{equation}
is empty for each $\varepsilon>0$.

In addition to the usual limit, there are other concepts of the limit. In 1915, A. Denjoy introduced the concept of approximate limit: Let $a\in\mathbb{R}$ and $f:\mathbb{R}\to\mathbb{R}$ be a function. The real number $L$ is said to be the approximate limit of the function $f$ if
$$
\lim\limits_{\delta\to 0^{+}}\frac{\left|\left\{x\in\mathbb{R}:\left|f\left(x\right)-L\right|\geq\varepsilon\right\}\cap(a-\delta,a+\delta)\right|}{2\delta}=0
$$
for each $\varepsilon>0$ \cite{denjoy}, and denoted by
\begin{equation}\label{eq2}
app\lim\limits_{x\to a}f\left(x\right)=L.
\end{equation}
This concept can also be given as follows. For each $\varepsilon>0$, there exists $\delta>0$ such that the set
$$
\mathcal{L}\left(f,\mathbb{R},a,L,\delta,\varepsilon\right)=\left\{x\in(a-\delta,a+\delta)\backslash\left\{a\right\}:\left|f\left(x\right)-L\right|\geq\varepsilon\right\}
$$
has zero density at point $a$, where the "density" of a measurable set $E$ at point $a$ is given by the following limit (if any)
\begin{equation}\label{eq3}
\lim\limits_{\delta\to 0^{+}}\frac{\left|E\cap(a-\delta,a+\delta)\right|}{2\delta}
\end{equation}
and $\left|\cdot\right|$ denotes the Lebesgue measure of a set \cite{pap}. Approximate limit has the properties of linearity, multiplicativity, and decomposition which means that if the function $f$ holds (\ref{eq2}), then
$$
f\left(x\right)=g\left(x\right)+h\left(x\right)
$$
for each $x\in\mathbb{R}$, where
$$
\lim\limits_{x\to a}g\left(x\right)=L
$$
and the set $\left\{x\in\mathbb{R}:h\left(x\right)\ne0\right\}$ has zero density at point $a$.

This paper investigates what happens when the set given in \ref{eq1} is not empty or not of zero density, but is taken in a different form. More specifically, we investigate what the definition of limit gives us when this set is finite, of empty limit points, countable or of zero measure, and whether a new definition of limit can be given in this way. When a new limit definition is given, it is usually examined whether it is equivalent to previous limit definitions.  We also examine whether the definitions we give are equivalent to previous limit definitions. We distinguish the equivalent ones and show why the different ones are different. Finally, we examine the properties of our new limit concepts such as existence, uniqueness, additivity and multiplicativity.

\section{Main Results}\label{sec2}

In this section, we give the types of limits defined before, the types of limits we define and the relationship among them. We denote the types of limit we will use in this study by $T_{1},T_{2},\dots$ which mean type1, type2,... First, we give the classical concept of limit and denote it by $T_{1}$.
\begin{definition}
Let $A\subset\mathbb{R}$, $f:A\to\mathbb{R}$ be a function, $L\in\mathbb{R}$ and $a$ be a limit point of $A$. If for each real $\varepsilon>0$, there exists a real $\delta_{\varepsilon}>0$ such that the condition
$$
\left\{x\in\left((a-\delta_{\varepsilon},a+\delta_{\varepsilon})\backslash\left\{a\right\}\right)\cap{A}:\left|f\left(x\right)-L\right|\geq\varepsilon\right\}=\varnothing
$$
holds, then $L$ is called the limit of the function $f$ at point $a$. In this study, we denote classical limit by
$$
T_{1}\lim\limits_{x\to{a}}f\left(x\right)=L.
$$
\end{definition}
Now, we give the definition of the approximate limit from our point of view.
\begin{definition}
Let $A\subset\mathbb{R}$ be a measurable set, $f:A\to\mathbb{R}$ be a function, and $a,L\in\mathbb{R}$. If for each real $\varepsilon>0$, there exists a real $\delta_{\varepsilon}>0$ such that the set
$$
\left\{x\in\left((a-\delta_{\varepsilon},a+\delta_{\varepsilon})\backslash\left\{a\right\}\right)\cap{A}:\left|f\left(x\right)-L\right|\geq\varepsilon\right\}
$$
has zero density at point $a$ (see (\ref{eq3})), then $L$ is called the approximate limit of the function $f$ at point $a$. In this study, we denote this limit by
$$
T_{2}\lim\limits_{x\to{a}}f\left(x\right)=L.
$$
\end{definition}
If a function has $T_{1}$ limit and it is $L$, then the number $L$ is also $T_{2}$ limit of this function. However, the converse is not true. Below is an example. 
\begin{example}\label{exmp1}
Consider the function $D:\mathbb{R}\to\mathbb{R}$
$$
D\left(x\right)=\left\{
\begin{array}{ll}
1,&x\in\mathbb{Q}\\
0,&x\notin\mathbb{Q}
\end{array}
\right.
$$
known as Dirichlet's function. This function has no $T_{1}$ limit, while $T_{2}$ limit is 0 since the set of Rational numbers has zero measure. 
\end{example}
The limits $T_{1}$ and $T_{2}$ are known concepts. Now we give new limit concepts and argue whether each of them is equivalent to $T_{1}$ or $T_{2}$ or not. 
\begin{definition}
Let $A\subset\mathbb{R}$, $f:A\to\mathbb{R}$ be a function and $a,L\in\mathbb{R}$. If for each real $\varepsilon>0$, there exists a real $\delta_{\varepsilon}>0$ such that the set
$$
\left\{x\in\left((a-\delta_{\varepsilon},a+\delta_{\varepsilon})\backslash\left\{a\right\}\right)\cap{A}:\left|f\left(x\right)-L\right|\geq\varepsilon\right\}
$$
is finite, then we say $L$ is $T_{3}$ limit of the function $f$ at point $a$. In this study, we denote this limit by
$$
T_{3}\lim\limits_{x\to{a}}f\left(x\right)=L.
$$
\end{definition}
\begin{theorem}\label{thm1}
Let $A\subset\mathbb{R}$, $f:A\to\mathbb{R}$ be a function and $a,L\in\mathbb{R}$. Then,
$$
T_{1}\lim\limits_{x\to{a}}f\left(x\right)=L\Leftrightarrow T_{3}\lim\limits_{x\to{a}}f\left(x\right)=L
$$
i.e., $T_{1}$ and $T_{3}$ limits are equivalent to each other.
\end{theorem}
\begin{proof}
Let $T_{1}\lim\limits_{x\to{a}}f\left(x\right)=L$. Then, the condition
$$
\forall\varepsilon>0, \exists\delta_{\varepsilon}>0: \left\{x\in\left((a-\delta_{\varepsilon},a+\delta_{\varepsilon})\backslash\left\{a\right\}\right)\cap{A}:\left|f\left(x\right)-L\right|\geq\varepsilon\right\}=\varnothing
$$
holds. Since $\varnothing$ is also a finite set, then $L$ is also $T_{3}$ limit of the function $f$ at point $a$.

Conversely, let $T_{3}\lim\limits_{x\to{a}}f\left(x\right)=L$. For each $\varepsilon>0$, there exists $\delta_{\varepsilon}>0$ such that the set
$$
\left\{x\in\left((a-\delta_{\varepsilon},a+\delta_{\varepsilon})\backslash\left\{a\right\}\right)\cap{A}:\left|f\left(x\right)-L\right|\geq\varepsilon\right\}
$$
is finite. Denote the above finite set by $\left\{t_{1},t_{2},\dots t_{n}\right\}$. One can easily see that $t_{k}\ne a$ for each $k=\overline{1,n}$. By taking $\delta_{\varepsilon}^{\prime}=\min\left\{\delta_{\varepsilon},\left|t_{1}-a\right|,\left|t_{2}-a\right|,\dots \left|t_{n}-a\right|\right\}$, we have
$$
\forall\varepsilon>0, \exists\delta_{\varepsilon}^{\prime}>0: \left\{x\in\left((a-\delta_{\varepsilon}^{\prime},a+\delta_{\varepsilon}^{\prime})\backslash\left\{a\right\}\right)\cap{A}:\left|f\left(x\right)-L\right|\geq\varepsilon\right\}=\varnothing.
$$
This completes the proof.
\end{proof}
\begin{definition}
Let $A\subset\mathbb{R}$, $f:A\to\mathbb{R}$ be a function and $a,L\in\mathbb{R}$. If for each real $\varepsilon>0$, there exists a real $\delta_{\varepsilon}>0$ such that the condition
$$
\left\{x\in\left((a-\delta_{\varepsilon},a+\delta_{\varepsilon})\backslash\left\{a\right\}\right)\cap{A}:\left|f\left(x\right)-L\right|\geq\varepsilon\right\}^{\prime}=\varnothing
$$
holds (i.e., that set has no limit points), then we say $L$ is $T_{4}$ limit of the function $f$ at point $a$. In this study, we denote this limit by
$$
T_{4}\lim\limits_{x\to{a}}f\left(x\right)=L.
$$
\end{definition}
As known from fundamental analysis, a bounded set of reals is finite if and only if it has no limit points. Thus, it can be concluded that  $T_{3}$ and $T_{4}$ limits are equivalent to each other, i.e., they are equivalent to the classical limit ($T_{1}$ limit). Consequently, they are not new limit concepts. Theorem \ref{thm1} and the following discussions show that the methods $T_{3}$ and $T_{4}$ fail to define a novel concept of limit. Now, we present the concept of the $T_{5}$ limit as a new limit.
\begin{definition}
Let $A\subset\mathbb{R}$, $f:A\to\mathbb{R}$ be a function and $a,L\in\mathbb{R}$. If for each real $\varepsilon>0$, there exists a real $\delta_{\varepsilon}>0$ such that the set
$$
\left\{x\in\left((a-\delta_{\varepsilon},a+\delta_{\varepsilon})\backslash\left\{a\right\}\right)\cap{A}:\left|f\left(x\right)-L\right|\geq\varepsilon\right\}
$$
is countable, then we say $L$ is $T_{5}$ limit of the function $f$ at point $a$. In this study, we denote this limit by
$$
T_{5}\lim\limits_{x\to{a}}f\left(x\right)=L.
$$
\end{definition}
If $L\in\mathbb{R}$ is the $T_{1}$ limit of a function, then it is also $T_{5}$ limit of that function. 
Besides, every countable set has zero measure. So, if $L\in\mathbb{R}$ is the $T_{5}$ limit of a function, then it is also $T_{2}$ limit of that function. In this study, we denote that relation among that limit concepts by
\begin{equation}\label{eq4}
T_{1}=T_{3}=T_{4}\subset T_{5}\subset T_{2}.
\end{equation}

Now, we give two examples to show that  $T_{5}$ limit is different from the previous ones. The first example is the Dirichlet function given in Example \ref{exmp1}. There is no $T_{1}$ limit this function anywhere. However, its $T_{5}$ limit is 0 at each point of Reals since the set of Rational numbers is countable. Before giving the second example, it is important to review the Cantor set and a significant property.

The Cantor set $C$ is obtained by dividing the interval $\left[0,1\right]$ into 3 equal parts, discarding the middle part and continuing this process infinitely:
$$
\begin{array}{c}
C_{0}=\left[0,1\right],\\
C_{1}=\left[0,\frac{1}{3}\right]\cup\left[\frac{2}{3},1\right],\\
C_{2}=\left[0,\frac{1}{9}\right]\cup\left[\frac{2}{9},\frac{1}{3}\right]\cup\left[\frac{2}{3},\frac{7}{9}\right]\cup\left[\frac{8}{9},1\right],\\
\vdots\\
C=\bigcap\limits_{n=0}^{\infty}C_{n}.
\end{array}
$$
$C$ is an uncountable set with zero measure \cite{balcireel}.
\begin{lemma}
    The set $C\cap\left(0,\delta\right)$ is uncountable for each positive real $\delta$.
\end{lemma}
\begin{proof}
    Since $\delta$ is positive, there exists $n\in\mathbb{N}$ such that $\frac{1}{3^{n}}<\delta$. The interval $\left[0,\frac{1}{3^{n}}\right]$ is the first part of $C_{n}$ which we use to define the Cantor set. The function $\phi\left(x\right)=3^{n}x$ is a bijection between $C\cap\left[0,\frac{1}{3^{n}}\right]$ and $C$. Thus, the set $C\cap\left[0,\frac{1}{3^{n}}\right]$ is uncountable, and consequently, $C\cap\left[0,\delta\right)$ which includes $C\cap\left[0,\frac{1}{3^{n}}\right]$ is uncountable. A singleton does not change the cardinality of any set. So, $C\cap\left(0,\delta\right)$ is also uncountable.
\end{proof}
Now, we can give the second example.
\begin{example}\label{exmp2}
Consider the characteristic function of the Cantor set on $\mathbb{R}$, $\chi_{C}:\mathbb{R}\to\mathbb{R}$
$$
\chi_{C}\left(x\right)=\left\{
\begin{array}{ll}
1,&x\in C,\\
0,&x\notin C.
\end{array}
\right.
$$
and choose $a=L=0$. By the relation
$$
\forall\varepsilon>0, 0\le\lim\limits_{\delta\to0^{+}}\frac{\left|\left\{x\in\left(-\delta,\delta\right)\setminus\left\{0\right\}:\left|\chi_{C}\left(x\right)\right|\ge\varepsilon\right\}\right|}{2\delta}\le\lim\limits_{\delta\to0^{+}}\frac{\left|C\right|}{2\delta}=\lim\limits_{\delta\to0^{+}}\frac{0}{2\delta}=0,
$$
we have
$$
T_{2}\lim\limits_{x\to{0}}\chi_{C}\left(x\right)=0.
$$
On the other hand, since the set
$$
\left\{x\in\left(-\delta,\delta\right)\setminus\left\{0\right\}:\left|\chi_{C}\left(x\right)\right|\ge1\right\}=C\cap\left(0,\delta\right)
$$
is uncountable for each $\delta>0$, then
$$
T_{5}\lim\limits_{x\to{0}}\chi_{C}\left(x\right)\ne0
$$
or does not exist.
\end{example}
Example \ref{exmp1} and Example \ref{exmp2} show that the concept of $T_{5}$ limit is different from $T_{1}$ and $T_{2}$. So, the relation (\ref{eq4}) turns into the following
\begin{equation}\label{eq5}
T_{1}=T_{3}=T_{4}\subsetneqq T_{5}\subsetneqq T_{2}.
\end{equation}

Now, we give a characterization for the uniqueness of $T_{5}$ limit.
\begin{theorem}\label{thm2}
Let $A\subset\mathbb{R}$ and $a\in\mathbb{R}$. Every function that has a $T_5$ limit at point $a$ has a unique limit if and only if the set $\left(\left(a-\delta,a+\delta\right)\setminus\left\{a\right\}\right)\cap A$ is uncountable for each $\delta>0$.
\end{theorem}
\begin{proof}
Assume that there exists a real $\delta_{0}>0$ such that $\left(\left(a-\delta_{0},a+\delta_{0}\right)\setminus\left\{a\right\}\right)\cap A$ is countable, $f:A\to\mathbb{R}$ is a function and $L$ is an arbitrary real number. Since
$$
\begin{array}{c}
\mathcal{L}\left(f,A,a,L,\delta_{0},\varepsilon\right)=\left\{x\in\left((a-\delta_{0},a+\delta_{0})\backslash\left\{a\right\}\right)\cap{A}:\left|f\left(x\right)-L\right|\geq\varepsilon\right\}\\
\subset\left(\left(a-\delta_{0},a+\delta_{0}\right)\setminus\left\{a\right\}\right)\cap A,
\end{array}
$$
then the set $\mathcal{L}\left(f,A,a,L,\delta_{0},\varepsilon\right)$ is countable for each $\varepsilon>0$. The last shows that all the real $L$ are $T_{5}$ limits of the function $f$. Conversely, assume that the set $\left(\left(a-\delta,a+\delta\right)\setminus\left\{a\right\}\right)\cap A$ is uncountable for each real $\delta>0$. We show that a function has a unique $T_{5}$ limit if it exists. Let $L_{1}$ and $L_{2}$ be $T_{5}$ limit of a function $f$ from $A$ to $\mathbb{R}$ at point $a$. There exist $\delta_{\varepsilon}^{\prime},\delta_{\varepsilon}^{\prime\prime}>0$ such that the sets $\left\{x\in\left((a-\delta_{\varepsilon}^{\prime},a+\delta_{\varepsilon}^{\prime})\backslash\left\{a\right\}\right)\cap{A}:\left|f\left(x\right)-L_{1}\right|\geq\varepsilon\right\}$ and $\left\{x\in\left((a-\delta_{\varepsilon}^{\prime\prime},a+\delta_{\varepsilon}^{\prime\prime})\backslash\left\{a\right\}\right)\cap{A}:\left|f\left(x\right)-L_{2}\right|\geq\varepsilon\right\}$ are countable for each $\varepsilon>0$.  By choosing $\delta_{\varepsilon}=\min\left\{\delta_{\varepsilon}^{\prime},\delta_{\varepsilon}^{\prime\prime}\right\}$, we have the sets
$$
\mathcal{L}\left(f,A,a,L_{1},\delta_{\varepsilon},\varepsilon\right)=\left\{x\in\left((a-\delta_{\varepsilon},a+\delta_{\varepsilon})\backslash\left\{a\right\}\right)\cap{A}:\left|f\left(x\right)-L_{1}\right|\geq\varepsilon\right\}
$$
and
$$
\mathcal{L}\left(f,A,a,L_{2},\delta_{\varepsilon},\varepsilon\right)=\left\{x\in\left((a-\delta_{\varepsilon},a+\delta_{\varepsilon})\backslash\left\{a\right\}\right)\cap{A}:\left|f\left(x\right)-L_{2}\right|\geq\varepsilon\right\}
$$
are countable. Then, $\mathcal{L}\left(f,A,a,L_{1},\delta_{\varepsilon},\varepsilon\right)\cup\mathcal{L}\left(f,A,a,L_{2},\delta_{\varepsilon},\varepsilon\right)$ is also countable. Since the set $\left(\left(a-\delta_{\varepsilon},a+\delta_{\varepsilon}\right)\setminus\left\{a\right\}\right)\cap A$ is uncountable, then
$$
\left(\left(\left(a-\delta_{\varepsilon},a+\delta_{\varepsilon}\right)\setminus\left\{a\right\}\right)\cap A\right)\setminus\left(\mathcal{L}\left(f,A,a,L_{1},\delta_{\varepsilon},\varepsilon\right)\cup\mathcal{L}\left(f,A,a,L_{2},\delta_{\varepsilon},\varepsilon\right)\right)
$$
is uncountable. Let $x_{\varepsilon}$ be an element of the above set. Then, the inequalities
$$
\left|f\left(x_{\varepsilon}\right)-L_{1}\right|<\varepsilon,\ \left|f\left(x_{\varepsilon}\right)-L_{2}\right|<\varepsilon
$$
hold. So, the relation
$$
\left|L_{1}-L_{2}\right|\le\left|f\left(x_{\varepsilon}\right)-L_{1}\right|+\left|f\left(x_{\varepsilon}\right)-L_{2}\right|<2\varepsilon
$$
is satisfied for each $\varepsilon>0$. The last shows that $L_{1}=L_{2}$. This completes the proof.
\end{proof}
We now prove a decomposition theorem for $T_{5}$ limit.
\begin{theorem}\label{thm3}
Let $A\subset\mathbb{R}$, $f:A\to\mathbb{R}$ be a function and $a,L\in\mathbb{R}$. Then, $L$ is $T_{5}$ limit of the function $f$ at point $a$ if and only if there exist two functions $g,h:A\to\mathbb{R}$ and $\delta_{0}>0$ such that $f\left(x\right)=g\left(x\right)+h\left(x\right)$ for each $x\in A$, $\lim\limits_{x\to a}g\left(x\right)=L$ and the set $\left\{x\in\left(\left(a-\delta_{0},a+\delta_{0}\right)\setminus\left\{a\right\}\right)\cap A:h\left(x\right)\ne0\right\}$ is countable.
\end{theorem}
\begin{proof}
Assume that there exist two functions $g,h:A\to\mathbb{R}$ and there exists a real $\delta_{0}>0$ such that $f\left(x\right)=g\left(x\right)+h\left(x\right)$ for each $x\in A$, $\lim\limits_{x\to a}f\left(x\right)=L$ and the set $\left\{x\in\left(\left(a-\delta_{0},a+\delta_{0}\right)\setminus\left\{a\right\}\right)\cap A:h\left(x\right)\ne0\right\}$ is countable. Given $\varepsilon>0$. Since $\lim\limits_{x\to a}g\left(x\right)=L$, then
$$
\exists\delta_{\varepsilon}^{*}>0:\ \left\{x\in\left((a-\delta_{\varepsilon}^{*},a+\delta_{\varepsilon}^{*})\backslash\left\{a\right\}\right)\cap{A}:\left|g\left(x\right)-L\right|\geq\varepsilon\right\}=\varnothing.
$$
Let's take $\delta_{\varepsilon}=\min\left\{\delta_{0},\delta_{\varepsilon}^{*}\right\}$. Since the relation
$$
\begin{array}{c}
\left\{x\in\left((a-\delta_{\varepsilon},a+\delta_{\varepsilon})\backslash\left\{a\right\}\right)\cap{A}:\left|f\left(x\right)-L\right|\geq\varepsilon\right\}\\
\subset\left\{x\in\left(\left(a-\delta_{\varepsilon},a+\delta_{\varepsilon}\right)\cap A\right)\setminus\left\{a\right\}:h\left(x\right)\ne0\right\}
\end{array}
$$
holds and the set $\left\{x\in\left(\left(a-\delta_{\varepsilon},a+\delta_{\varepsilon}\right)\cap A\right)\setminus\left\{a\right\}:h\left(x\right)\ne0\right\}$ is countable, then the set $\left\{x\in\left((a-\delta_{\varepsilon},a+\delta_{\varepsilon})\backslash\left\{a\right\}\right)\cap{A}:\left|f\left(x\right)-L\right|\geq\varepsilon\right\}$ is also countable. Hence, we have
$$
T_{5}\lim\limits_{x\to{a}}f\left(x\right)=L.
$$
Conversely, assume that the relation above holds. There exists $\delta_{\varepsilon}>0$ such that the set
$$
\left\{x\in\left((a-\delta_{\varepsilon},a+\delta_{\varepsilon})\backslash\left\{a\right\}\right)\cap{A}:\left|f\left(x\right)-L\right|\geq\varepsilon\right\}
$$
is countable for each $\varepsilon>0$. Assume that $\varepsilon=\frac{1}{n}$ and $n$ is a natural number. The set
$$
\mathcal{L}\left(f,A,a,L,\delta_{\frac{1}{n}},\frac{1}{n}\right)=\left\{x\in\left((a-\delta_{\frac{1}{n}},a+\delta_{\frac{1}{n}})\backslash\left\{a\right\}\right)\cap{A}:\left|f\left(x\right)-L\right|\geq\frac{1}{n}\right\}
$$
is countable for each $n\in\mathbb{N}$. We denote the above set by $A_{n}$. We define $g:A\to\mathbb{R}$ by the equality
$$
g\left(x\right)=\left\{
\begin{array}{ll}
L,&\exists k\in\mathbb{N}: x\in A_{k},\\
f\left(x\right),&\text{otherwise.}
\end{array}
\right.
$$
Given $\varepsilon>0$. There exists $n\in\mathbb{N}$ such that $\frac{1}{n}<\varepsilon$. Assume that $x\in A$ and $0<\left|x-a\right|<\delta_{\frac{1}{n}}$. If $\exists k\in\mathbb{N}$ such that $x\in A_{k}$, $\left|g\left(x\right)-L\right|=0<\varepsilon$; otherwise, $\left|g\left(x\right)-L\right|=\left|f\left(x\right)-L\right|<\frac{1}{n}<\varepsilon$. This shows the relation $\lim\limits_{x\to a}g\left(x\right)=L$ holds.

We also define $h:A\to\mathbb{R}$ by $h\left(x\right)=f\left(x\right)-g\left(x\right)$. Then, the relation
$$
\left\{x\in\left(\left(a-\delta_{\frac{1}{n}},a+\delta_{\frac{1}{n}}\right)\setminus\left\{a\right\}\right)\cap A:h\left(x\right)\ne0\right\}\subset\bigcup\limits_{k=1}^{\infty}A_{k}
$$
holds for each $n\in\mathbb{N}$. So, the set $\left\{x\in\left(\left(a-\delta_{\frac{1}{n}},a+\delta_{\frac{1}{n}}\right)\setminus\left\{a\right\}\right)\cap A:h\left(x\right)\ne0\right\}$ is countable. Let's take $\delta_{0}=\sup\limits_{n\in\mathbb{N}}\delta_{\frac{1}{n}}$. Since
$$
\begin{array}{c}
\left\{x\in\left(\left(a-\delta_{0},a+\delta_{0}\right)\setminus\left\{a\right\}\right)\cap A:h\left(x\right)\ne0\right\}\\
=\bigcup\limits_{n=1}^{\infty}\left\{x\in\left(\left(a-\delta_{\frac{1}{n}},a+\delta_{\frac{1}{n}}\right)\setminus\left\{a\right\}\right)\cap A:h\left(x\right)\ne0\right\},
\end{array}
$$
then the set $\left\{x\in\left(\left(a-\delta_{0},a+\delta_{0}\right)\setminus\left\{a\right\}\right)\cap A:h\left(x\right)\ne0\right\}$ is also countable. This completes the proof.
\end{proof}
Some algebraic properties of $T_{5}$ limit are given in the following theorem
\begin{theorem}\label{thm4}
Let $A\subset\mathbb{R}$, $f_{1},f_{2}:A\to\mathbb{R}$ be two functions and $a,\lambda,L_{1},L_{2}\in\mathbb{R}$. If $T_{5}\lim\limits_{x\to{a}}f_{1}\left(x\right)=L_{1}$ and $T_{5}\lim\limits_{x\to{a}}f_{2}\left(x\right)=L_{2}$, then the relations
\begin{enumerate}
    \item $T_{5}\lim\limits_{x\to{a}}\left(f_{1}\left(x\right)+f_{2}\left(x\right)\right)=L_{1}+L_{2}$,
    \item $T_{5}\lim\limits_{x\to{a}}\left(f_{1}\left(x\right)f_{2}\left(x\right)\right)=L_{1}L_{2}$,
    \item $T_{5}\lim\limits_{x\to{a}}\frac{f_{1}\left(x\right)}{f_{2}\left(x\right)}=\frac{L_{1}}{L_{2}}$, $\left(L_{2}\ne0\right)$
    \item $T_{5}\lim\limits_{x\to{a}}\left(\lambda f_{1}\left(x\right)\right)=\lambda L_{1}$
\end{enumerate}
hold.
\end{theorem}
\begin{proof}
Since the relations $T_{5}\lim\limits_{x\to{a}}f_{1}\left(x\right)=L_{1}$ and $T_{5}\lim\limits_{x\to{a}}f_{2}\left(x\right)=L_{2}$ hold, then there exist four functions $g_{1},g_{2},h_{1},h_{2}$ from $A$ to $\mathbb{R}$ and there exists a real $\delta_{0}>0$ such that $f_{1}\left(x\right)=g_{1}\left(x\right)+h_{1}\left(x\right)$, $f_{2}\left(x\right)=g_{2}\left(x\right)+h_{2}\left(x\right)$ for each $x\in A$; $\lim\limits_{x\to{a}}g_{1}\left(x\right)=L_{1}$, $\lim\limits_{x\to{a}}g_{2}\left(x\right)=L_{2}$; $\left\{x\in\left(\left(a-\delta_{0},a+\delta_{0}\right)\setminus\left\{a\right\}\right)\cap A:h_{1}\left(x\right)\ne0\right\}$ and $\left\{x\in\left(\left(a-\delta_{0},a+\delta_{0}\right)\setminus\left\{a\right\}\right)\cap A:h_{2}\left(x\right)\ne0\right\}$ are countable.
\begin{enumerate}
    \item The equalities $f_{1}\left(x\right)+f_{2}\left(x\right)=g_{1}\left(x\right)+g_{2}\left(x\right)+h_{1}\left(x\right)+h_{2}\left(x\right)$ for each $x\in A$ and $\lim\limits_{x\to a}\left(g_{1}\left(x\right)+g_{2}\left(x\right)\right)=L_{1}+L_{2}$ are valid. Since the proposition
$$
h_{1}\left(x\right)+h_{2}\left(x\right)\ne0\Rightarrow h_{1}\left(x\right)\ne0\lor h_{2}\left(x\right)\ne0
$$
holds, then the relation
$$
\begin{array}{c}
\left\{x\in\left(\left(a-\delta_{0},a+\delta_{0}\right)\setminus\left\{a\right\}\right)\cap A:h_{1}\left(x\right)+h_{2}\left(x\right)\ne0\right\}\\
\subset\left\{x\in\left(\left(a-\delta_{0},a+\delta_{0}\right)\setminus\left\{a\right\}\right)\cap A:h_{1}\left(x\right)\ne0\right\}\\
\cup\left\{x\in\left(\left(a-\delta_{0},a+\delta_{0}\right)\setminus\left\{a\right\}\right)\cap A:h_{2}\left(x\right)\ne0\right\}
\end{array}
$$
is also holds. So, the set
$$
\left\{x\in\left(\left(a-\delta_{0},a+\delta_{0}\right)\setminus\left\{a\right\}\right)\cap A:h_{1}\left(x\right)+h_{2}\left(x\right)\ne0\right\}
$$
is countable. By Theorem \ref{thm3}, we have
$$
T_{5}\lim\limits_{x\to{a}}\left(f_{1}\left(x\right)+f_{2}\left(x\right)\right)=L_{1}+L_{2}.
$$
\item The equalities $f_{1}\left(x\right)f_{2}\left(x\right)=g_{1}\left(x\right)g_{2}\left(x\right)+g_{1}\left(x\right)h_{2}\left(x\right)+g_{2}\left(x\right)h_{1}\left(x\right)+h_{1}\left(x\right)h_{2}\left(x\right)$ for each $x\in A$ and $\lim\limits_{x\to a}g_{1}\left(x\right)g_{2}\left(x\right)=L_{1}L_{2}$ are valid. We define a new function $h$ by $h\left(x\right)=h_{1}\left(x\right)g_{2}\left(x\right)+h_{2}\left(x\right)\left(g_{1}\left(x\right)+h_{1}\left(x\right)\right)$. Since the relation
$$
\begin{array}{c}
\left\{x\in\left(\left(a-\delta_{0},a+\delta_{0}\right)\setminus\left\{a\right\}\right)\cap A:h\left(x\right)\ne0\right\}\\
\subset\left\{x\in\left(\left(a-\delta_{0},a+\delta_{0}\right)\setminus\left\{a\right\}\right)\cap A:h_{1}\left(x\right)\ne0\right\}\\
\cup\left\{x\in\left(\left(a-\delta_{0},a+\delta_{0}\right)\setminus\left\{a\right\}\right)\cap A:h_{2}\left(x\right)\ne0\right\}
\end{array}
$$
holds, then the set $\left\{x\in\left(\left(a-\delta_{0},a+\delta_{0}\right)\setminus\left\{a\right\}\right)\cap A:h\left(x\right)\ne0\right\}$ is countable. By Theorem \ref{thm3}, we have
$$
T_{5}\lim\limits_{x\to{a}}f_{1}\left(x\right)f_{2}\left(x\right)=L_{1}L_{2}.
$$
\item First, let's prove that the equality $T_{5}\lim\limits_{x\to{a}}\frac{1}{f\left(x\right)}=\frac{1}{L}$ is satisfied provided that  $T_{5}\lim\limits_{x\to{a}}f\left(x\right)=L$ and $L\ne0$. By Theorem \ref{thm3}, there exist two functions $g,h$ and there exists real $\delta_{0}>0$ such that $f\left(x\right)=g\left(x\right)+h\left(x\right)$ for each $x\in A$, $\lim\limits_{x\to a}f\left(x\right)=L$ and the set
$$
\left\{x\in\left(\left(a-\delta_{0},a+\delta_{0}\right)\setminus\left\{a\right\}\right)\cap A:h\left(x\right)\ne0\right\}
$$
is countable. Since the relations
$$
\begin{array}{c}
\frac{1}{f\left(x\right)}=\frac{1}{g\left(x\right)}-\frac{h\left(x\right)}{g\left(x\right)\left(g\left(x\right)+h\left(x\right)\right)},\\
\lim\limits_{x\to{a}}\frac{1}{g\left(x\right)}=\frac{1}{L},\\
\left\{x\in\left(\left(a-\delta_{0},a+\delta_{0}\right)\setminus\left\{a\right\}\right)\cap A:-\frac{h\left(x\right)}{g\left(x\right)\left(g\left(x\right)+h\left(x\right)\right)}\ne0\right\}
\\
\subset\left\{x\in\left(\left(a-\delta_{0},a+\delta_{0}\right)\setminus\left\{a\right\}\right)\cap A:h\left(x\right)\ne0\right\}
\end{array}
$$
hold, then we have
$$
T_{5}\lim\limits_{x\to{a}}\frac{1}{f\left(x\right)}=\frac{1}{L}
$$
by Theorem \ref{thm3} again. Consequently,
$$
\begin{array}{c}
T_{5}\lim\limits_{x\to{a}}\frac{f_{1}\left(x\right)}{f_{2}\left(x\right)}=T_{5}\lim\limits_{x\to{a}}f_{1}\left(x\right)\frac{1}{f_{2}\left(x\right)}\\ =\left(T_{5}\lim\limits_{x\to{a}}f_{1}\left(x\right)\right)\left(T_{5}\lim\limits_{x\to{a}}\left(x\right)\frac{1}{f_{2}\left(x\right)}\right)=\frac{L_{1}}{L_{2}}.
\end{array}
$$
\item Directly obtained from the second equality of Theorem \ref{thm4}.
\end{enumerate}
\end{proof}
Now, we present the concept of the $T_{6}$ limit as a new limit.
\begin{definition}
Let $A\subset\mathbb{R}$, $f:A\to\mathbb{R}$ be a function and $a,L\in\mathbb{R}$. If for each real $\varepsilon>0$, there exists a real $\delta_{\varepsilon}>0$ such that
$$
\left|\left\{x\in\left((a-\delta_{\varepsilon},a+\delta_{\varepsilon})\backslash\left\{a\right\}\right)\cap{A}:\left|f\left(x\right)-L\right|\geq\varepsilon\right\}\right|=0,
$$
then we say $L$ is $T_{6}$ limit of the function $f$ at point $a$, where $\left|K\right|$ denotes the Lebesgue measure of the set $K$. In this study, we denote this limit by
$$
T_{6}\lim\limits_{x\to{a}}f\left(x\right)=L.
$$
\end{definition}
Any countable set has a measure of zero. So, we have
\begin{equation}\label{eq6}
T_{5}\lim\limits_{x\to{a}}f\left(x\right)=L\Rightarrow T_{6}\lim\limits_{x\to{a}}f\left(x\right)=L,
\end{equation}
but the converse is not true. The characteristic function of the Cantor set given in Example \ref{exmp2} has a 0 limit in the sense of $T_{6}$ at point $a=0$, but it has no $T_{5}$ limit at the same point.

Now, we investigate the relation between $T_{2}$ and $T_{6}$ limits. If a set has zero measure, then the density of the set at any point of $\mathbb{R}$ is zero. So, we have
\begin{equation}\label{eq7}
T_{6}\lim\limits_{x\to{a}}f\left(x\right)=L\Rightarrow T_{2}\lim\limits_{x\to{a}}f\left(x\right)=L.
\end{equation}
We now give an example to show that the converse is not true.
\begin{example}\label{exmp3}
Consider the set
$$
\Omega=\bigcup\limits_{n=1}^{\infty}\left[\frac{1}{n}-\frac{1}{2^{n}},\frac{1}{n}\right).
$$
First, we show that the set $\Omega$ has zero density at point 0, i.e.,
\begin{equation}\label{eq8}
\lim\limits_{\delta\to0^{+}}\frac{\left|\Omega\cap\left(-\delta,\delta\right)\right|}{2\delta}=0.
\end{equation}
Given $\delta>0$. There exists a natural number $n_{\delta}$ such that the equality $\frac{1}{n_{\delta}+1}\le\delta<\frac{1}{n_{\delta}}$ holds. Then, the relation
$$
\frac{\left|\Omega\cap\left(-\delta,\delta\right)\right|}{2\delta}\le\frac{\left|\bigcup\limits_{k=n_{\delta}}^{\infty}\left[\frac{1}{k}-\frac{1}{2^{k}},\frac{1}{k}\right)\right|}{2\frac{1}{n_{\delta}+1}}=\frac{\left|\sum\limits_{k=n_{\delta}}^{\infty}\frac{1}{2^{k}}\right|}{\frac{2}{n_{\delta}+1}}=\frac{n_{\delta}+1}{2^{n_{\delta}+2}}
$$
holds. This proves the relation (\ref{eq8}).

Second, we prove that the inequality
\begin{equation}\label{eq9}
\left|\Omega\cap\left(-\delta,\delta\right)\right|>0
\end{equation}
for each $\delta>0$. There exists a natural number $n_{\delta}$ such that the equality $\frac{1}{n_{\delta}}<\delta$ holds for each real $\delta>0$. So, the relations
$$
\bigcup\limits_{k=n_{\delta}+1}^{\infty}\left[\frac{1}{k}-\frac{1}{2^{k}},\frac{1}{k}\right)\subset\Omega\cap\left(-\delta,\delta\right)
$$
and
$$
\left|\bigcup\limits_{k=n_{\delta}+1}^{\infty}\left[\frac{1}{k}-\frac{1}{2^{k}},\frac{1}{k}\right)\right|=\sum\limits_{k=n_{\delta}+1}^{\infty}\frac{1}{2^{k}}=\frac{1}{2^{n_{\delta}+2}}>0
$$
hold. This proves the relation (\ref{eq9}).

Now, we consider the characteristic function of $\chi_{\Omega}$ from $\mathbb{R}$ to $\mathbb{R}$. Assume that $a=L=0$. Given $\varepsilon>0$. Since,
$$
\lim\limits_{\delta\to0^{+}}\frac{\left|\left\{x\in\left(-\delta,\delta\right)\setminus\left\{0\right\}:\left|\chi_{\Omega}\left(x\right)\right|\ge\varepsilon\right\}\right|}{2\delta}=\lim\limits_{\delta\to0^{+}}\frac{\left|\Omega\cap\left(-\delta,\delta\right)\setminus\left\{0\right\}\right|}{2\delta}=0,
$$
then,
$$
T_{2}\lim\limits_{x\to0}\chi_{\Omega}\left(x\right)=0.
$$
On the other hand, since
$$
\left|\left\{x\in\left(-\delta,\delta\right)\setminus\left\{0\right\}:\left|\chi_{\Omega}\left(x\right)\right|\ge\varepsilon\right\}\right|=\left|\Omega\cap\left(-\delta,\delta\right)\setminus\left\{0\right\}\right|>0
$$
for each $\varepsilon,\delta>0$, then
$$
T_{6}\lim\limits_{x\to0}\chi_{\Omega}\left(x\right)\ne0.
$$
\end{example}
By (\ref{eq6}), (\ref{eq7}) and Example \ref{exmp3}, the relation (\ref{eq5}) turns into
$$
T_{1}=T_{3}=T_{4}\subsetneqq T_{5}\subsetneqq T_{6}\subsetneqq T_{2}.
$$
Now, we give the fundamental properties of $T_{6}$ limit. Note that the following theorems can be proved similarly to Theorem \ref{thm2}, \ref{thm3} and \ref{thm4}.
\begin{theorem}\label{thm5}
Let $A\subset\mathbb{R}$ be a measurable set and $a\in\mathbb{R}$. Every function that has a $T_6$ limit at point $a$ has a unique limit if and only if $\left|\left(\left(a-\delta,a+\delta\right)\setminus\left\{a\right\}\right)\cap A\right|>0$ for each $\delta>0$.
\end{theorem}
\begin{theorem}\label{thm6}
Let $A\subset\mathbb{R}$, $f:A\to\mathbb{R}$ be a function and $a,L\in\mathbb{R}$. Then, $L$ is $T_{6}$ limit of the function $f$ at point $a$ if and only if there exist two functions $g,h:A\to\mathbb{R}$ and $\delta_{0}>0$ such that $f\left(x\right)=g\left(x\right)+h\left(x\right)$ for each $x\in A$, $\lim\limits_{x\to a}g\left(x\right)=L$ and $\left|\left\{x\in\left(\left(a-\delta_{0},a+\delta_{0}\right)\setminus\left\{a\right\}\right)\cap A:h\left(x\right)\ne0\right\}\right|=0$.
\end{theorem}
\begin{theorem}\label{thm7}
Let $A\subset\mathbb{R}$, $f_{1},f_{2}:A\to\mathbb{R}$ be two functions and $a,\lambda,L_{1},L_{2}\in\mathbb{R}$. If $T_{6}\lim\limits_{x\to{a}}f_{1}\left(x\right)=L_{1}$ and $T_{6}\lim\limits_{x\to{a}}f_{2}\left(x\right)=L_{2}$, then the relations
\begin{enumerate}
    \item $T_{6}\lim\limits_{x\to{a}}\left(f_{1}\left(x\right)+f_{2}\left(x\right)\right)=L_{1}+L_{2}$,
    \item $T_{6}\lim\limits_{x\to{a}}\left(f_{1}\left(x\right)f_{2}\left(x\right)\right)=L_{1}L_{2}$,
    \item $T_{6}\lim\limits_{x\to{a}}\frac{f_{1}\left(x\right)}{f_{2}\left(x\right)}=\frac{L_{1}}{L_{2}}$, $\left(L_{2}\ne0\right)$
    \item $T_{6}\lim\limits_{x\to{a}}\left(\lambda f_{1}\left(x\right)\right)=\lambda L_{1}$
\end{enumerate}
hold.
\end{theorem}








\end{document}